\documentclass[11pt, twoside, leqno, a4]{amsart}  
\usepackage{lipsum}
\usepackage{amsfonts}
\usepackage{graphicx}
\usepackage{epstopdf}
\usepackage{algorithmic}
\usepackage{calligra}
\usepackage[dvipsnames]{xcolor}
\usepackage{amsfonts,amsmath,amsthm,amssymb}
\usepackage{enumitem}
\usepackage{mathtools}
\usepackage{hyperref}
\usepackage{autonum}
\usepackage{hhline}
\usepackage{array}
\usepackage{diagbox}
\usepackage{tcolorbox}
\usepackage{mdframed}
\usepackage{multicol}
\usepackage{graphicx}
\usepackage{subcaption}
\usepackage{moreverb}
\usepackage{bbm}
\usepackage[margin=1.4in]{geometry}
\usepackage{todonotes}
\usepackage{scalerel,amssymb}
\allowdisplaybreaks
\usepackage{mathrsfs}  
\usepackage{lineno}
\usepackage{todonotes}
\usepackage{tikz}
\usepackage{pgfplots}
\pgfplotsset{compat=1.7}
\usetikzlibrary{arrows.meta}
\usepackage[numbers,sort&compress]{natbib}
\usepackage{appendix}
\DeclareMathAlphabet{\mathdutchcal}{U}{dutchcal}{m}{n}
\SetMathAlphabet{\mathdutchcal}{bold}{U}{dutchcal}{b}{n}
\DeclareMathAlphabet{\mathdutchbcal}{U}{dutchcal}{b}{n}

\newcommand{\dx}{\, \textup{d} x}

\newcommand{\R}{\mathbb{R}} 
\newcounter{rownumber}

\usepackage{scalerel,stackengine}
\stackMath
\newcommand\reallywidehat[1]{%
\savestack{\tmpbox}{\stretchto{%
  \scaleto{%
    \scalerel*[\widthof{\ensuremath{#1}}]{\kern-.6pt\bigwedge\kern-.6pt}%
    {\rule[-\textheight/2]{1ex}{\textheight}}
  }{\textheight}%
}{0.5ex}}%
\stackon[1pt]{#1}{\tmpbox}%
}
\newtheorem{theorem}{Theorem}
\newtheorem{lemma}{Lemma}
\newtheorem{proposition}{Proposition}
\newtheorem*{assumption*}{Assumptions on the memory kernel}

\numberwithin{lemma}{section}
\numberwithin{proposition}{section}
\numberwithin{theorem}{section}
\numberwithin{equation}{section}
\makeatletter
\newcommand{\leqnomode}{\tagsleft@true}
\newcommand{\reqnomode}{\tagsleft@false}
\makeatother
      
\title[Global well-posedness  for JMGT equation]{Global well-posedness of the Cauchy problem for the Jordan--Moore--Gibson--Thompson equation with arbitrarily large higher-order Sobolev norms
} 
\subjclass[2010]{35L75, 35G25}

\keywords{nonlinear acoustics,  asymptotic behavior, JMGT equation}

\author[B. Said-Houari ]{\bfseries Belkacem Said-Houari}

\address{  
	Department of Mathematics\\ College of Sciences\\ University of
Sharjah, P. O. Box: 27272 \\ Sharjah, United Arab Emirates}
\email{bhouari@sharjah.ac.ae}

\begin{document}  
\vspace*{-4mm} 
\maketitle        
\vspace*{-4mm}  
	
\vspace*{3mm}
\begin{abstract}    
   In this paper, we consider  the 3D  Jordan--Moore--Gibson--Thompson equation arising in nonlinear acoustics. First, we prove that the solution exists globally in time provided that    the lower order  Sobolev norms of the initial data are considered to be small, while the higher-order norms can be arbitrarily large. This improves some available results in the literature. Second, we prove a new decay estimate for the linearized model and  removing the $L^1$-assumption on the initial data.   The proof of this decay estimate is based on the high-frequency and  low-frequency decomposition of the solution together with an interpolation inequality related to Sobolev spaces with negative order.                        
\end{abstract}              
\vspace*{2mm}       
                
\section{Introduction}  
In this paper, we consider the nonlinear Jordan--Moore--Gibson--Thompson (JMGT) 
equation: 
\begin{subequations}
\label{Main_problem}
\begin{equation}
\tau u_{ttt}+u_{tt}-c^{2}\Delta u-\beta \Delta u_{t}=\frac{\partial }{%
\partial t}\left( \frac{1}{c^{2}}\frac{B}{2A}(u_{t})^{2}+|\nabla
u|^{2}\right) ,  \label{MGT_1}
\end{equation}%
where $x\in \mathbb{R}^{3}$ (Cauchy problem in  3D) and $t>0$ and $u=u(x,t)$ denotes the acoustic velocity potential.  We consider the
initial conditions
\begin{eqnarray}
u(t=0)=u_{0},\qquad u_{t}(t=0)=u_{1}\qquad u_{tt}(t=0)=u_{2}.
\label{Initial_Condition}
\end{eqnarray}
\end{subequations}

The JMGT equation with different types of damping mechanisms has received a substantial amount of attention in recent years and this because of its wide applications in  medicine and industry~\cite{maresca2017nonlinear, he2018shared, melchor2019damage, duck2002nonlinear}.

Equation \eqref{MGT_1} is an alternative model to the classical  Kuznetsov equation 
\begin{equation}\label{Kuznt}
u_{tt}-c^{2}\Delta u-\beta \Delta u_{t}=\frac{\partial }{%
\partial t}\left( \frac{1}{c^{2}}\frac{B}{2A}(u_{t})^{2}+|\nabla
u|^{2}\right),
\end{equation}
where $u=u(x,t)$ represents the acoustic velocity potential for $x \in \R^3$ and $t>0$; see~\cite{kuznetsov1971equations}. The equation \eqref{Kuznt} can be obtained as an approximation of the governing equations of fluid mechanics by means of asymptotic expansions in powers of small parameters; see~\cite{crighton1979model,Coulouvrat_1992, kuznetsov1971equations, kaltenbacher2007numerical}.
The constants $c>0$ and $\beta >0$ are the speed and the diffusivity of sound, respectively. The parameter of nonlinearity $B/A$ arises in the Taylor expansion of the variations of pressure in a medium in terms of the variations of density; cf.~\cite{beyer1960parameter}. The extra term $\tau u_{ttt}$ appearing in \eqref{MGT_1} is due to the replacement of the Fourier law of heat conduction in the equation of the   conservation of energy by the
 Cattaneo (or Maxwell--Cattaneo)
 law  which
accounts for finite speed of propagation of the heat transfer and
eliminates the paradox of infinite speed of propagation for pure heat
conduction associated with the Fourier law.

The starting point  of the nonlinear analysis lies in the results for the linearization
\begin{equation}\label{MGT_2_1}
\tau u_{ttt}+ \alpha u_{tt}-c^{2}\Delta u-\beta \Delta u_{t}=0.
\end{equation}
This equation is known as the Moore--Gibson--Thompson equation (although, as mentioned in \cite{bucci2019feedback}, this model originally appears in the work of Stokes~\cite{stokes1851examination}).	Interestingly, equation \eqref{MGT_2_1} also arises in viscoelasticity theory under the name of \emph{standard linear model} of vicoelasticity;~see \cite{Gorain_2010} and references given therein.

Equation \eqref{MGT_2_1} has been extensively studied lately; see, for example \cite{bucci2019feedback, bucci2019regularity, Chen_Palmieri_1, conejero2015chaotic,Kal_Las_Mar,Lizama_Zamorano_2019, Trigg_et_al, PellSaid_2019_1, P-SM-2019}, and the references therein.  
In particular in \cite{Kal_Las_Mar} (see also \cite%
{kaltenbacher2012well}), the authors considered the linear equation in bounded domains
\begin{equation}  \label{MGT_22}
\tau u_{ttt}+\alpha u_{tt}+c^{2}\mathcal{A} u+\beta \mathcal{A} u_{t}=0,
\end{equation}
where $\mathcal{A}$ is a positive self-adjoint operator. They proved that when the diffusivity of the sound is strictly
positive ($\beta>0$), the linear dynamics is described by a strongly continuous
semigroup, which is exponentially stable provided the dissipativity
condition $\gamma:=\alpha-\tau c^2/\beta>0$ is fulfilled.
  The study of the controllability properties of the MGT type equations can be found for instance in \cite{bucci2019feedback, Lizama_Zamorano_2019}.
The MGT equation in $\R^N$ with a power source nonlinearity of the form $|u|^p$  has been considered in \cite{Chen_Palmieri_1}  where some blow up results have been shown for the critical case $\tau c^2=\alpha\beta$.

The MGT and JMGT equations with a memory term have been also investigated recently. For the  MGT with memory, the reader is refereed to \cite{Bounadja_Said_2019,Liuetal._2019,dell2016moore} and to \cite{lasiecka2017global,nikolic2020mathematical,Nikolic_SaidHouari_2} for the JMGT with memory.  
The singular limit problem when $\tau\rightarrow 0$ has been rigorously justified  in \cite{KaltenbacherNikolic}. The authors in \cite{KaltenbacherNikolic} showed that in bounded domain,    the limit of \eqref{MGT_1}  as $\tau \rightarrow 0$ leads to the Kuznetsov equation (i.e., Eq \eqref{MGT_1} with $\tau=0$).
Concerning the large time asymptotic stability, the author and Pellicer showed in   \cite{PellSaid_2019_1}   the following  decay estimate of the solution of the Cauchy problem associated to \eqref{MGT_2_1}: 
        \begin{align}\label{marta decay}
        \Vert V(t)\Vert _{L^{2}(\R^{N})}\lesssim  (1+t)^{-N/4}\big(\Vert V_{0}\Vert _{L^{1}(\R^{N})}+\Vert V_{0}\Vert _{L^{2}(\R^{N})}\big) .
        \end{align}
 with  $V=(u_{t}+\tau u_{tt},\nabla(u+\tau u_{t}),\nabla u_{t})$. The method used to prove \eqref{marta decay} is based on a pointwise energy estimates in the Fourier space together with suitable asymptotic integral estimates.  The decay rate in \eqref{marta decay} under the $L^1$ assumption on the initial data seems sharp since it matches the decay rate of the heat kernel.

The global well posedness and large time behaviour of the solution to the Cauchy problem associated to the nonlinear 3D model  \eqref{Main_problem} has been recently investigated in \cite{Racke_Said_2019}. More precisely,   under  the assumption $0<\tau c^2<\beta$ and  by using the
contraction mapping theorem in appropriately chosen spaces, the authors  showed a local
existence result in some appropriate functional spaces.  In addition  using a bootstrap argument,  a global existence result and decay estimates for the solution with small initial
data were  proved. The decay estimate obtained in \cite{Racke_Said_2019} agrees with the one of the linearized models  given in \cite{PellSaid_2019_1}.  

Our main goal in this paper is first to improve the global existence result in \cite{Racke_Said_2019} by removing the smallness assumption on the higher-order Sobolev norms. More precisely, we only assume  the lower-order  Sobolev norms of initial data to be small, while the higher-order norms can be arbitrarily large. To achieve this,  and inspired by \cite{Guo_Wang_2012} we use different estimates than those in \cite{Racke_Said_2019} in order to control the nonlinearity in more precise way.  
 Second, as in \eqref{marta decay}, to prove the decay rate of the solution, it is common to take initial data to be  in $L^1(\R^n)$ and combine this with energy estimates in $H^s(\R^n), \, s\geq 0$ . However, this may create some difficulties, especially for the nonlinear problems since in some situations it is important to propagate this assumption of the  $L^1$-initial data over time, which is  not always the case.  Hence, it is very important to replace  this $L^1$ space by  $\dot{H}^{-\gamma},\, \gamma>0$ which is an $L^2$-type space. In fact, we prove (see Theorem \ref{Theorem_Decay}) instead of \eqref{MGT_2_1}, the following decay estimate:  
\begin{equation}\label{Decay_Negative_Norm}
\Vert V(t)\Vert_{L^2}\lesssim (1+t)^{-\gamma},\quad \gamma>0.    
\end{equation}
provided that the initial data are in $\dot{H}^{-\gamma}(\R^N)\cap L^2(\R^N)$.  The proof of the decay estimate \eqref{Decay_Negative_Norm} is based on the high-frequency and  low-frequency decomposition of the solution together with an interpolation inequality related to Sobolev spaces with negative order (see Lemma \ref{Lemma_gamma_Interpo} below). In fact, we prove that the low-frequency part of the solution behaves similarly to the solution of the heat equation 
\begin{equation}\label{Heat_Eq}
\partial_t\psi-\Delta \psi=0, \quad \text{in}\quad \R^3
\end{equation}
and hence, we recover  the decay rate in \cite{Guo_Wang_2012} for equation  \eqref{Heat_Eq} in the Sobolev space of a negative order. For the  high-frequency part, we show that it follows the decay rate of the ``toy" model 
\begin{equation}
\partial_t\psi+\psi=0,\quad  \text{in}\quad \R^3, 
\end{equation}
which is an exponential decay rate. 

The rest of this paper is organized as follows:  Section~\ref{Sec:Preliminaries} contains the necessary theoretical preliminaries, which allow us to rewrite the equation with the corresponding initial data as a  first-order Cauchy problem and define the main energy norm with the associated dissipative norm. We also recall a local well-posedness result from \cite{Racke_Said_2019}.  In Section~\ref{Sec:Mair_Result}, we state and discuss  our main result. Section~\ref{Section_Global_Existence} is devoted to the proof of the global existence result.  
 Section~\ref{Sec: Decay_Linearized} is dedicated to the proof of the decay estimates of the linearized problem. In Appendix \ref{Appendix_Usefull_Ineq},  we present the    Gagliardo--Nirenberg inequality together with some Sobolev interpolation inequalities that we used in the proofs. 
 
\subsection{Notation} Throughout the paper, the constant $C$ denotes a generic positive constant
that does not depend on time, and can have different values on different occasions.
We often write $f \lesssim g$ where there exists a constant $C>0$, independent of parameters of interest such that   $f\leq C g$ and we analogously define $f\gtrsim g$. We sometimes use the notation $f\lesssim_\alpha g$ if we want to emphasize that the implicit constant depends on some parameter $\alpha$. The notation $f\approx g$ is used when there exists a constant $C>0$ such that $C^{-1}g\leq f\leq Cg$.

\section{Preliminaries} \label{Sec:Preliminaries}
We  rewrite the right-hand side of equation \eqref{MGT_1} in the form
\begin{equation}
\frac{\partial }{\partial t}\left( \frac{1}{c^{2}}\frac{B}{2A}%
(u_{t})^{2}+|\nabla u|^{2}\right) =\frac{1}{c^{2}}\frac{B}{A}%
u_{t}u_{tt}+2\nabla u \cdot \nabla u_{t},
\end{equation}
and introduce the new variables
\begin{equation}
v=u_{t}\qquad \text{ and }\qquad w=u_{tt},
\end{equation}%
and without loss of generality, we assume from now on that
$
c=1.
$
Then equation \eqref{MGT_1} can be rewritten as the following first-order system 
\begin{subequations}\label{Main_System_First_Order}
\begin{equation}
\left\{
\begin{array}{ll}
u_{t}=v, &  \\
v_{t}=w, &  \\
\tau w_{t}=\Delta u+\beta \Delta v-w+\dfrac{B}{A}vw+2\nabla u \cdot \nabla v , &
\end{array}%
\right.  \label{System_New}
\end{equation}
with the initial data \eqref{Initial_Condition} rewritten as 
\begin{eqnarray}  \label{Initial_Condition_2}
u(t=0)=u_0,\qquad v(t=0)=v_0,\qquad w(t=0)=w_0.
\end{eqnarray}
\end{subequations}

Let  $\mathbf{U}=(u,v,w)$ be the solution of \eqref{Main_System_First_Order}.  In order to state our main result, for $k \geq 0$, we introduce the energy $\mathcal{E}_k[\mathbf{U}](t)$ of order $k$ and the corresponding dissipation $\mathcal{D}_k\mathbf{U}(t)$ as
follows:%

\begin{equation}\label{Weighted_Energy}
\begin{aligned}
\mathcal{E}^2_{k}[\mathbf{U}](t)=&\,\sup_{0\leq \sigma\leq t}\Big(\big\Vert \nabla^k(v+\tau
w)(\sigma)\big\Vert _{H^{1}}^{2}+\big\Vert \Delta \nabla^k v(\sigma)\big\Vert
_{L^{2}}^{2}+\big\Vert \nabla^{k+1} v(\sigma)\big\Vert _{L^{2}}^{2}\Big.%
\vspace{0.2cm}  \notag \\
 \Big.&+\big\Vert \Delta \nabla^k(u+\tau v)(\sigma)\big\Vert
_{L^{2}}^{2}+\big\Vert \nabla^{k+1} (u+\tau v)(\sigma)\big\Vert
_{L^{2}}^{2}+\Vert \nabla^kw(\sigma)\Vert_{L^2}^2\Big),  
\end{aligned}
\end{equation}
and
\begin{equation}\label{Dissipative_weighted_norm_1}
\mathcal{D}^{2}_k[\mathbf{U}](t) =\int_0^t\mathscr{D}^2_k(\sigma)\textup{d}\sigma 
\end{equation}
with 
\begin{equation}
\begin{aligned}
\mathscr{D}^2_k[\mathbf{U}](t)=&\, \begin{multlined}[t]\,\left(\big\Vert \nabla^{k+1}
v(t)\big\Vert _{L^{2}}^{2}+\big\Vert \Delta \nabla^k
v(t)\big\Vert
_{L^{2}}^{2}+ \Vert \nabla^kw (t)\Vert_{L^2}^2\right. \\
\left.+\big\Vert \Delta \nabla^k\left( u+\tau v\right)(t) \big\Vert
_{L^{2}}^{2} + \big\Vert \nabla^{k+1} (v+\tau w)(t)\big\Vert _{L^{2}}^{2}%
\right). \end{multlined}
\end{aligned}
\end{equation}

Let $V=(v +\tau w , \nabla( u +\tau v),\nabla v)$. It is clear that for all $t\geq 0$, 
\begin{equation}\label{Equiv_E_V}
\begin{aligned}
\mathcal{E}_k^{2}[\mathbf{U}](t)\approx&\, \Vert  \nabla^k V(t)\Vert_{L^2}^2 +\Vert \nabla^{k+1} V(t)\Vert_{L^2}^2+\Vert  w(t)\Vert_{L^2}^2.
\end{aligned} 
\end{equation}
For a positive integer
$s\geq 1$ that will be fixed later on, we define
\begin{eqnarray}\label{E_s_D_s_Def}
\mathrm{E}_s^2[\mathbf{U}](t)=\sum_{k=0}^s \mathcal{E}_k^2[\mathbf{U}](t)\qquad \text{and}\qquad \mathrm{D}_s^2[\mathbf{U}](t)=\sum_{k=0}^s \mathcal{D}_k^2[\mathbf{U}](t).
\end{eqnarray}
To introduce energies with negative indices, we first define the operator $\Lambda^\gamma$ for $\gamma\in \R $ by 
\begin{equation}
\Lambda^\gamma f(x)=\int_{\R^3} |\xi|^\gamma \hat{f}(\gamma) 2^{2i\pi x\cdot \xi} \, \textup{d}\xi,
\end{equation}
where $\hat{f}$ is the Fourier transform of $f$.  The homogenous Sobolev spaces $\dot{H}^\gamma$ consist of all $f$ for which 
\begin{equation}
\Vert f\Vert_{\dot{H}^\gamma}=\Vert\Lambda^\gamma f\Vert_{L^2}=\Vert |\xi|^\gamma \hat{f}\Vert_{L^2}
\end{equation}
is finite. 
We can then define the energy functional associated to the negative Sobolev spaces as
\begin{equation}
\begin{aligned}
\mathcal{E}_{-\gamma}^{2}[\mathbf{U}](t)=&\,\sup_{0\leq \sigma\leq t}\Big(\left\Vert \Lambda^{-\gamma}(v+\tau
w)(\sigma)\right\Vert _{H^{1}}^{2}+\left\Vert \Lambda^{-\gamma}\Delta  v(\sigma)\right\Vert
_{L^{2}}^{2}+\left\Vert \Lambda^{-\gamma}\nabla v(\sigma)\right\Vert _{L^{2}}^{2}\Big.%
\vspace{0.2cm}   \\
& \Big.+\left\Vert \Lambda^{-\gamma}\Delta (u+\tau v)(\sigma)\right\Vert
_{L^{2}}^{2}+\left\Vert \Lambda^{-\gamma}\nabla (u+\tau v)(\sigma)\right\Vert
_{L^{2}}^{2}+\Vert \Lambda^{-\gamma}w(\sigma)\Vert_{L^2}^2\Big).  \label{Weighted_Energy_gamma}
\end{aligned}
\end{equation}
The associated dissipative term is given by
\begin{eqnarray}
\hspace{0.8cm} \mathcal{D}^{2}_{-\gamma}[\mathbf{U}](t) &=&\int_0^t\Big(\Big\Vert \Lambda^{-\gamma}\nabla
v(\sigma)\Big\Vert _{L^{2}}^{2}+\left\Vert \Lambda^{-\gamma} \Delta
v(\sigma)\right\Vert
_{L^{2}}^{2}+ \Vert \Lambda^{-\gamma}w (\sigma)\Vert_{L^2}^2\Big.  \notag \\
&&\Big.+\left\Vert \Lambda^{-\gamma}\Delta \left( u+\tau v\right)(\sigma) \right\Vert
_{L^{2}}^{2} + \left\Vert \Lambda^{-\gamma}\nabla (v+\tau w)(\sigma)\right\Vert _{L^{2}}^{2}%
\Big)\textup{d}\sigma.  \label{Dissipative_weighted_norm_gamma}
\end{eqnarray}
 
In the following theorem, we recall a local well-posedness result obtained in \cite{Racke_Said_2019}. 
 \begin{theorem}[see Theorem 1.2 in \cite{Racke_Said_2019}]
\label{Local_Ex_Theorem} Assume that $0<\tau<\beta$ and let  $s>\frac{5}{2} $. Let $\mathbf{U}_0=(u_0,v_0,w_0)^T$ be such that
\begin{eqnarray}  \label{Upsilon_s_Assum}
\mathrm{E}_{s}^2[\mathbf{U}](0) &=&\left\Vert v_0+\tau w_0\right\Vert
_{H^{s+1}}^{2}+\left\Vert \Delta v_0\right\Vert _{H^{s}}^{2}+\left\Vert
\nabla v_0\right\Vert _{H^{s}}^{2}  \notag \\
&&+\left\Vert \Delta (u_0+\tau v_0)\right\Vert _{H^{s}}^{2}+\left\Vert
\nabla (u_0+\tau v_0)\right\Vert _{H^{s}}^{2}+\left\Vert
w_0\right\Vert _{H^{s}}^{2}\leq\tilde{\delta}_0
\end{eqnarray}  
for some $\tilde{\delta}_0>0$. Then, there exists a small time $%
T=T(\mathrm{E}_s(0))>0$ such that problem \eqref{Main_problem} has a unique
solution $u$ on $[0,T) \times \mathbb{R}^{3}$ satisfying
\begin{eqnarray*}
 \mathrm{E}_s^2[\mathbf{U}](T)+\mathrm{D}_s^2[\mathbf{U}](T)\leq C_{\tilde{\delta}_0},
\end{eqnarray*}
where $\mathrm{E}_s^2[\mathbf{U}](T)$ and $\mathrm{D}_s^2[\mathbf{U}](T)$ are given in \eqref{E_s_D_s_Def}, determining the regularity of $u$, and $C_{%
\tilde{\delta}_0}$ is a positive constant depending on $\tilde{\delta}_0$.
\end{theorem}
\section{Main results}\label{Sec:Mair_Result}
\noindent In this section, we state and discuss  our main results. The global existence result is stated in Theorem \ref{Main_Theorem}, while the decay estimate for the linearized problem is given in Theorem \ref{Theorem_Decay}. 
\begin{theorem}\label{Main_Theorem}
Assume that $0<\tau<\beta$  and let $s \geq 3$ be an integer. Let $s_0=\max\{[2s/3]+1,[s/2]+2\}\leq m\leq s $. 
   Assume that $%
u_{0},v_0,w_0 $ are such that 
$\mathrm{E}_{s}[\mathbf{U}](0)<\infty$. 
Then there exists a small positive
constant $\delta ,$ such that if
\begin{equation}\label{Initial_Assumption_Samll}
\begin{aligned}
\mathrm{E}_{s_0}^2[\mathbf{U}](0) =&\,\left\Vert v_0+\tau w_0\right\Vert
_{H^{s_0+1}}^{2}+\left\Vert \Delta v_0\right\Vert _{H^{s_0}}^{2}+\left\Vert
\nabla v_0\right\Vert _{H^{s_0}}^{2}  \\
&+\left\Vert \Delta (u_0+\tau v_0)\right\Vert _{H^{s_0}}^{2}+\left\Vert  
\nabla (u_0+\tau v_0)\right\Vert _{H^{s_0}}^{2}+\left\Vert
 w_0\right\Vert _{H^{s_0}}^{2}\leq \delta,
\end{aligned}
\end{equation}
then problem  \eqref{Main_problem} admits a unique global-in-time solution satisfying 
\begin{equation}\label{Main_Energy_Estimate}
\mathrm{E}^2_{m}[\mathbf{U}](t)+\mathrm{D}_{m}^2[\mathbf{U}](t)
\leq \mathrm{E}^2_{m}[\mathbf{U}](0), \qquad t \geq 0,
\end{equation}
where $s_0\leq m\leq s$. 
\end{theorem}

In the following theorem, we state a decay estimate of the solution of the linearized problem associated to \eqref{Main_System_First_Order}. 
\begin{theorem}\label{Theorem_Decay}
  Let $\mathbf{U}$ be the solution of the linearized problem associated to \eqref{Main_System_First_Order}.  Assume that $0<\tau<\beta$.  Let $\gamma>0$ and let $\mathbf{U}(0)$ be such that  $%
\mathcal{E}^2_{-\gamma}[\mathbf{U}](0)< \infty$. Then, it holds that 
\begin{equation}\label{Boundedness_E_gamma}
\mathcal{E}^2_{-\gamma}[\mathbf{U}](t)+\mathcal{D}_{-\gamma}^2[\mathbf{U}](t)
\leq \mathcal{E}^2_{-\gamma}[\mathbf{U}](0).
\end{equation}
In addition,  the following decay estimate of the linearized problem hold:  
\label{Decay}
\begin{equation}\label{Decay_1}
\Vert V(t)\Vert_{L^2}\lesssim_{C_0} (1+t)^{-\gamma}.   
\end{equation}
Here $C_0$ is a positive constant that depends on the initial data, but is independent of $t$.
\end{theorem} 
\subsection{Discussion of the main result}
 Before moving onto the proof, we briefly discuss the statements made above in Theorems \ref{Main_Theorem} and \ref{Theorem_Decay}. 
 \begin{itemize}
\item Similarly to the result in \cite{Guo_Wang_2012}, we only assume  the lower-order  Sobolev norms of initial data to be small, while the higher-order norms can be arbitrarily large. This improves the recent result of \cite[Theorem 1.1]{Racke_Said_2019} where all the norms up to order $s$ are assumed to be small. To do this, and inspired by \cite{Guo_Wang_2012}, we employ different techniques to tackle nonlinear terms rather than the usual commutator estimates. More precisely, we use Sobolev interpolation of the Gagliardo--Nirenberg inequality between higher-order and lower-order spatial derivatives to tackle the nonlinear terms. 
  
\item The decay rate for the linearized  equation obtained in \cite{PellSaid_2019_1} holds under the  assumption that the initial data $V_0\in L^1(\R^3)$. Theorem~\ref{Main_Theorem} does not require the initial data to be in $L^1(\R^3)$. Instead, we take the initial data to be in $H^{-\gamma}$, which is obvious due to \eqref{Boundedness_E_gamma}, that this norm is preserved.  This can be shown (under some restrictions on $\gamma$) to hold  also for the nonlinear problem. 
   However, it seems difficult to extend the decay estimate \eqref{Decay_1} to the nonlinear problem since the cut-off operators defined in  
\eqref{Cut-off_Operator} induces some commutators that are difficult to control by the lower frequency dissipative terms. The decay estimates of the nonlinear problem provided in \cite{Guo_Wang_2012} are  mainly based on an estimate of the form 
\begin{equation}
\frac{\textup {d}}{\textup {d}t}\Vert \nabla^\ell V(t)\Vert_{L^2}+\Vert \nabla^{\ell+1}V(t)\Vert_{L^2}\leq 0. 
\end{equation}
  Such estimate seems difficult to obtain in our situation due to the nature of our equation \eqref{Main_problem}. 

\item Theorem \ref{Theorem_Decay} holds for all $\gamma>0$.  The decay rate obtained in \cite{Guo_Wang_2012} is restricted to the case $\gamma\in [0,3/2)$, this restriction is needed to control the nonlinear terms.  
 
\end{itemize}

\section{Energy estimates}\label{Section_Global_Existence}  
The main goal of  this section is to use the energy method to derive the main estimates of the solution, which will be used to prove Theorem \ref{Main_Theorem}. In fact, we prove by a continuity argument that the energy $\mathrm{E}_{m}[\mathbf{U}](t)$ is uniformly bounded for all time if $\delta$ is sufficiently small. The main idea in the proof is to bound the nonlinear terms by $\mathrm{E}_{s_0}[\mathbf{U}](t)\mathrm{D}_{m}^2[\mathbf{U}](t) $ and get the estimate \eqref{Estimate_Main}. 
As a result,  if we prove that  $\mathrm{E}_{s_0}[\mathbf{U}](t)\leq \varepsilon$ provided that $\delta$ is sufficiently small, then we can absorb the last term in \eqref{Estimate_Main} by the left-hand side. 
To control the nonlinear terms, we do not use the commutator estimates as in \cite{Racke_Said_2019}, instead and inspired by \cite{Guo_Wang_2012}, we use Sobolev interpolation of the Gagliardo--Nirenberg inequality between higher-order and lower-order spatial derivatives. 

 Let $s_0$ be as in Theorem \ref{Main_Theorem}. We now use a bootstrap  argument to show that $\mathrm{E}_{s_0}[\mathbf{U}](t)$ is uniformly bounded.  
We recall that 
\begin{equation}
\mathrm{E}_{s_0}^2[\mathbf{U}](t)=\sum_{k=0}^{s_0} \mathcal{E}_k^2[\mathbf{U}](t). 
\end{equation}
We derive our estimates under the a priori assumption 
\begin{equation}\label{boot_strap_Assum}
\mathrm{E}_{s_0}^2[\mathbf{U}](t)\leq \varepsilon^2
\end{equation}
and show that 
\begin{equation}
\mathrm{E}_{s_0}^2[\mathbf{U}](t)\leq \frac{1}{2}\varepsilon^2.
\end{equation}
Hence, we deduce that $\mathrm{E}_{s_0}^2[\mathbf{U}](t)\leq \varepsilon^2$ provided that the initial energy $\mathrm{E}_{s_0}^2[\mathbf{U}](0)$  is small enough.  First, we have the following estimate  
\begin{proposition}[First-order energy estimate] \label{Prop:FirstOrderEE}
Let	$\mathrm{E}_{s_0}^2[\mathbf{U}](t)\leq \varepsilon^2$ for some $\varepsilon>0$ and a fixed integer $5/2<s_0<s$. Then
\begin{eqnarray}  \label{Main_Estimate_D_0}
\mathcal{E}_0^2[\mathbf{U}](t)+\mathcal{D}_0^2[\mathbf{U}](t)\lesssim  \mathcal{E}_0^2[\mathbf{U}](0) +\varepsilon
\mathcal{D}^2_0[\mathbf{U}](t).
\end{eqnarray}
\end{proposition}
\begin{proof}
According to~\cite[Estimate (2.39)]{Racke_Said_2019}, the following energy estimate holds: 
\begin{equation}  \label{Main_Estimate_D_0_1}
\begin{aligned}
\mathcal{E}_0^2[\mathbf{U}](t)+\mathcal{D}_0^2[\mathbf{U}](t)\lesssim&\, \mathcal{E}_0^2[\mathbf{U}](0) + \mathcal{E}%
_0[\mathbf{U}](t)\mathcal{D}^2_0[\mathbf{U}](t) \\&+M_0[\mathbf{U}](t)\mathcal{D}_0^2[\mathbf{U}](t),
\end{aligned}
\end{equation}
where
\begin{equation}
\begin{aligned}
M_0[\mathbf{U}](t)=& \,\sup_{0\leq s\leq t}\Big(\left\Vert v(s)\right\Vert _{L^{\infty
}}+\left\Vert (v+\tau w)(s)\right\Vert _{L^{\infty }}\Big.\\
&\Big.+\left\Vert \nabla
(u+\tau v)(s)\right\Vert _{L^{\infty }}+\Vert \nabla u(s)\Vert_{L^\infty}+\Vert \nabla v(s)\Vert_{L^\infty}\Big).
\end{aligned}	
\end{equation}
Using the Sobolev embedding theorem (recall that $s_0>5/2$) together with the assumption on $\mathrm{E}_{s_0}^2[\mathbf{U}](t)$ yields \[M_0[\mathbf{U}](t)+\mathcal{E}%
_0[\mathbf{U}](t)\lesssim \mathcal{E}%
_{s_0}[\mathbf{U}](t)\lesssim \varepsilon.\]
Plugging this inequality into \eqref{Main_Estimate_D_0_1} further yields the desired bound.  
\end{proof}

To prove a higher-order version of this energy estimate, we apply the operator $\nabla^k,\, k\geq 1$ to \eqref{System_New}. We obtain 
for $U:=\nabla^k u$, $V:=\nabla^k v$ and $W:=\nabla^k w$
\begin{equation}
\left\{
\begin{array}{ll}
\partial_t U=V,\vspace{0.2cm} &  \\
\partial_tV=W,\vspace{0.2cm} &  \\
\tau \partial_tW=\Delta U+\beta \Delta V-W+\nabla^k\Big(\dfrac{B}{A}vw+2\nabla u \cdot \nabla v\Big).  &
\end{array}%
\right.  \label{System_New_k}
\end{equation}
Let us also define the right-hand side functionals as
\begin{equation}\label{R_1_k}
\mathrm{R}^{(k)}=\nabla^k\Big(\dfrac{B}{A}vw+2\nabla u \cdot \nabla v\Big).  
\end{equation}
The following estimate holds; cf.~\cite[Estimate (2.50)]{Racke_Said_2019}.
\begin{proposition}[Higher-order energy estimate, \cite{Racke_Said_2019}] Under the assumptions of Proposition~\ref{Prop:FirstOrderEE}, for all $1\leq k\leq s$, it holds
\begin{eqnarray}  \label{E_I_Est}
&&\mathcal{E}^2_k[\mathbf{U}](t)+\mathcal{D}_k^2[\mathbf{U}](t)
\lesssim \mathcal{E}^2_k[\mathbf{U}](0)+\sum_{i=1}^5\int_0^t \mathrm{\mathbf{I}}_i^{(k)}(\sigma)d\sigma
\end{eqnarray} 
with 
\begin{equation} \label{I_terms}
\begin{aligned}
\mathrm{\mathbf{I}}_1^{(k)}=&\,\Big|\int_{\mathbb{R}^{3}}\mathrm{R}^{(k)}(t)\left( V+\tau W\right)
\dx\Big|,\qquad \mathrm{\mathbf{I}}_2^{(k)}=\,\Big|\int_{\mathbb{R}^{3}}\nabla \mathrm{R}^{(k)} \nabla (V+\tau W)\dx\Big|,\\
\mathrm{\mathbf{I}}_3^{(k)}=&\,\Big|\int_{\mathbb{R}^{3}}\mathrm{R}^{(k)}\Delta \left( U+\tau
V\right) \dx\Big|,\qquad \mathrm{\mathbf{I}}_4^{(k)}=\, \Big|\int_{\mathbb{R}^{3}}\nabla \mathrm{R}^{(k)}\nabla V\dx\Big|, \\
\mathrm{\mathbf{I}}_5^{(k)}=&\,\Big|\int_{\mathbb{R}^{3}}\mathrm{R}^{(k)}W\dx\Big|.
\end{aligned}
\end{equation}
\end{proposition}
Thus our proof reduces to estimating the right-hand side terms $\mathrm{\mathbf{I}}_1^{(k)},\dots, \mathrm{\mathbf{I}}_5^{(k)}$. This will be done through several lemmas (see Lemmas \ref{Lemma_I_1}--\ref{Lemma_I_5} below).  Inspired by \cite{Guo_Wang_2012}, we use a different method to handle the nonlinearities compared to~\cite{Racke_Said_2019}. In particular, we will make extensive use of the Gagliardo--Nirenberg inequality \eqref{Interpolation_inequality} and the Sobolev--Gagliardo--Nirenberg inequality  \eqref{Sobolev_Gagl_Ni_Interpolation_ineq_Main}, which will allow us to interpolate between higher-order and lower-order Sobolev norms and ``close" the nonlinear estimates. 

\indent We thus wish to show an estimate of the form 
\begin{equation}  \label{Estimate_Main}
\mathrm{E}^2_s[\mathbf{U}](t)+\mathrm{D}^2_s[\mathbf{U}](t)\lesssim  \mathrm{E}%
^2_s[\mathbf{U}](0)+\mathrm{E}_{s_0}[\mathbf{U}](t)\mathrm{D}^2_s[\mathbf{U}](t),
\end{equation}  
which improves the one stated  in \cite{Racke_Said_2019}, where $\mathrm{E}_{s}(t)$ replaces $\mathrm{E}_{s_0}(t)$ in \eqref{Estimate_Main}. 

\subsection{Estimates of the terms $\mathrm{\mathbf{I}}_i^{(k)},\, 1\leq i\leq 5$}

The goal of this section is to provide the appropriate estimates of the last term on the right-hand side of the estimate \eqref{E_I_Est}. 

\begin{lemma}[Estimate of $\mathrm{\mathbf{I}}_1^{(k)}$]\label{Lemma_I_1} For any $1\leq k\leq s$, it holds that 
\begin{equation}\label{I_1_Estimate}
\begin{aligned}
\mathrm{\mathbf{I}}_1^{(k)}\lesssim \varepsilon 
&\Big(\Vert \nabla^k v\Vert
_{L^{2}}^{2}+\Vert \nabla^{k+1} v\Vert
_{L^{2}}^{2}+\Vert
\nabla^{k+2}u\Vert _{L^{2}}^{2}+\Vert
\nabla^{k}w\Vert _{L^{2}}^{2} +\Vert \nabla^{k+1}(v+\tau w)\Vert_{L^2}^2\Big)\\
\lesssim& \,\varepsilon\big( \mathscr{D}^2_{k-1}[\mathbf{U}](t)+\mathscr{D}^2_k[\mathbf{U}](t)\big).
\end{aligned}
\end{equation}
\end{lemma}
\begin{proof}
Recall the definition of $R^{(k)}$ \eqref{R_1_k}. We have  \begin{equation}
\begin{aligned}
\mathrm{\mathbf{I}}_1^{(k)}=&\,\int_{\mathbb{R}^{3}}\left|\nabla^{k-1}\left(\dfrac{B}{A}vw+2\nabla u \cdot \nabla v\right)\nabla^{k+1}(v+\tau w)\right|
\dx\\
=&\,\dfrac{B}{A}\int_{\mathbb{R}^{3}} \Big|\sum_{0\leq \ell\leq k-1}C_{k-1}^\ell\nabla^{k-1-\ell} v\nabla^\ell w\nabla^{k+1}(v+\tau w)\Big|
\dx\\
&+2\int_{\R^3} \Big|\sum_{0\leq \ell\leq k-1}C_{k-1}^\ell\nabla^{k-1-\ell} \nabla u\nabla^\ell\nabla  v\nabla^{k+1}(v+\tau w)\Big|\dx
=:\,\mathrm{\mathbf{I}}_{1;1}^{(k)}+\mathrm{\mathbf{I}}_{1;2}^{(k)}.
\end{aligned}
\end{equation}
The term $\mathrm{\mathbf{I}}_{1;1}^{(k)}$ can be estimates as follows: 

 \begin{equation}\label{I_1_1_Estimate_1}
 \begin{aligned}
\mathrm{\mathbf{I}}_{1;1}^{(k)}\lesssim \sum_{0\leq \ell\leq k-1}\Vert \nabla^{k-1-\ell}v\Vert_{L^3}\Vert \nabla^\ell w\Vert_{L^6} \Vert \nabla^{k+1}(v+\tau w)\Vert_{L^2}. 
\end{aligned}
\end{equation}  
Employing the Gagliardo--Nirenberg inequality \eqref{Interpolation_inequality} yields
 \begin{equation}\label{Interpo_I_1_1}
\Vert \nabla^\ell w\Vert_{L^6}\lesssim \Vert
\nabla^{k}w\Vert _{L^{2}}^{\frac{1+\ell}{k}}\left\Vert w\right\Vert
_{L^{2}}^{1-\frac{1+\ell}{k}},\qquad 0\leq \ell\leq k-1.
\end{equation}
Now, by applying the Sobolev--Gagliardo--Nirenberg inequality \eqref{Sobolev_Gagl_Ni_Interpolation_ineq_Main}, we obtain 
 \begin{equation}\label{Interpo_I_1_2}
\Vert\nabla^{k-1-\ell} v\Vert_{L^3}\lesssim \left\Vert
\nabla^{m_0}v\right\Vert _{L^{2}}^{\frac{1+\ell}{k}}\Vert \nabla^k v\Vert
_{L^{2}}^{1-\frac{1+\ell}{k}},\qquad 0\leq \ell\leq k-1
\end{equation}
with 
\begin{equation}
\frac{1}{3}=\frac{k-1-\ell}{3}+\left( \frac{1}{2}-\frac{m_0}{3}\right)\frac{1+\ell}{k} + \left( \frac{1}{2}-\frac{k}{3}\right)\Big(1-\frac{1+\ell}{k}\Big).
\end{equation}
This relation implies 
\begin{equation}
m_0=\frac{k}{2(1+\ell)}\leq \frac{k}{2}\leq \frac{s-1}{2}. 
\end{equation}
It is clear that for $s_0\geq [(s-1)/2]+1$ we have \[\left\Vert
\nabla^{m_0}v(t)\right\Vert _{L^{2}}\lesssim \mathcal{E}_{s_0}(t).\] Hence, by plugging estimates \eqref{Interpo_I_1_1} and \eqref{Interpo_I_1_2} into \eqref{I_1_1_Estimate_1}, and making use of assumption \eqref{boot_strap_Assum} on $\mathcal{E}_{s_0}(t)$, we obtain
\begin{equation}\label{I_1_1_Estimate_2}
 \begin{aligned}
\mathrm{\mathbf{I}}_{1;1}^{(k)}\lesssim&\, \sum_{0\leq \ell\leq k-1}\left\Vert
\nabla^{m_0}v\right\Vert _{L^{2}}^{\frac{1+\ell}{k}}\Vert \nabla^k v\Vert
_{L^{2}}^{1-\frac{1+\ell}{k}}\Vert
\nabla^{k}w\Vert _{L^{2}}^{\frac{1+\ell}{k}}\left\Vert w\right\Vert
_{L^{2}}^{1-\frac{1+\ell}{k}} \Vert \nabla^{k+1}(v+\tau w)\Vert_{L^2}\\
\lesssim&\, \varepsilon \sum_{0\leq \ell\leq k-1}\Vert \nabla^k v\Vert
_{L^{2}}^{1-\frac{1+\ell}{k}}\Vert
\nabla^{k}w\Vert _{L^{2}}^{\frac{1+\ell}{k}} \Vert \nabla^{k+1}(v+\tau w)\Vert_{L^2}. 
\end{aligned}
\end{equation}
Young's inequality implies that 
\begin{equation}\label{I_1_1_Estimate}
\mathrm{\mathbf{I}}_{1;1}^{(k)}\lesssim\varepsilon \Big(\Vert \nabla^k v\Vert
_{L^{2}}^{2}+\Vert
\nabla^{k}w\Vert _{L^{2}}^{2} +\Vert \nabla^{k+1}(v+\tau w)\Vert_{L^2}^2\Big).
\end{equation}

 Now, we estimate $\mathrm{\mathbf{I}}_{1;2}^{(k)}$. We have
 \begin{equation}\label{I_1_2_Estimate_1}
 \begin{aligned}
\mathrm{\mathbf{I}}_{1;2}^{(k)}\lesssim \sum_{0\leq \ell\leq k-1}\Vert \nabla^{k-\ell}u\Vert_{L^3}\Vert \nabla^{\ell+1} v\Vert_{L^6} \Vert \nabla^{k+1}(v+\tau w)\Vert_{L^2}. 
\end{aligned}
\end{equation}
By employing again the Gagliardo--Nirenberg inequality, we infer
 \begin{equation}
\Vert \nabla^{\ell+1} v\Vert_{L^6}\lesssim \Vert  v\Vert_{L^2}^{1-\frac{\ell+2}{k+1}}\Vert  \nabla^{k+1}v\Vert_{L^2}^{\frac{\ell+2}{k+1}}. 
\end{equation}
We then also have by using the Sobolev--Gagliardo--Nirenberg inequality \eqref{Sobolev_Gagl_Ni_Interpolation_ineq_Main}, 
 
  \begin{equation}
\begin{aligned}
\Vert \nabla^{k-\ell}u\Vert_{L^3}\lesssim \left\Vert
\nabla^{m_1+1}u\right\Vert _{L^{2}}^{\frac{\ell+2}{1+k}}\Vert \nabla^{k+2} u\Vert
_{L^{2}}^{1-\frac{\ell+2}{1+k}},\qquad 0\leq \ell\leq k-1
\end{aligned}
\end{equation}
with 
\begin{equation}
\frac{1}{3}=\frac{k-\ell}{3} + \left( \frac{1}{2}-\frac{k+2}{3}\right)\Big(1-\frac{\ell+2}{1+k}\Big)+\left( \frac{1}{2}-\frac{m_1+1}{3}\right)\frac{\ell+2}{1+k}.
\end{equation}
This yields 
\begin{equation}
m_1=\frac{k+1}{2(2+\ell)}\leq \frac{k+1}{4}\leq \frac{s}{4}. 
\end{equation}
Thus for $s_0\geq [\frac{s}{4}]+1$,  it holds that \[\left\Vert
\nabla^{m_1+1}u(t)\right\Vert _{L^{2}}\lesssim \mathcal{E}_{s_0}[\mathbf{U}](t).\] Consequently, by making use of the assumption \eqref{boot_strap_Assum} and the fact  that  $\Vert v\Vert_{L^2}\lesssim \mathcal{E}_{s_0}[\mathbf{U}](t)$, we have
\begin{equation}\label{I_1_2_Estimate_2}
 \begin{aligned}
\mathrm{\mathbf{I}}_{1;2}^{(k)}\lesssim &\,\sum_{0\leq \ell\leq k-1}\Vert \nabla^{k-\ell}u\Vert_{L^3}\Vert \nabla^{\ell+1} v\Vert_{L^6} \Vert \nabla^{k+1}(v+\tau w)\Vert_{L^2}\Vert \nabla^{k+1}(v+\tau w)\Vert_{L^2}\\
 \lesssim &\,\sum_{0\leq \ell\leq k-1}\Vert  v\Vert_{L^2}^{1-\frac{\ell+2}{k+1}}\Vert  \nabla^{k+1}v\Vert_{L^2}^{\frac{\ell+2}{k+1}}\left\Vert
\nabla^{m_1+1}u\right\Vert _{L^{2}}^{\frac{2+\ell}{1+k}}\Vert \nabla^{k+2} u\Vert
_{L^{2}}^{1-\frac{\ell+2}{1+k}}\Vert \nabla^{k+1}(v+\tau w)\Vert_{L^2}\\
 \lesssim &\,\varepsilon \sum_{0\leq \ell\leq k-1}\Vert  \nabla^{k+1}v\Vert_{L^2}^{\frac{\ell+2}{k+1}}\Vert \nabla^{k+2} u\Vert
_{L^{2}}^{1-\frac{\ell+2}{1+k}}\Vert \nabla^{k+1}(v+\tau w)\Vert_{L^2}.
\end{aligned}
\end{equation}
Applying Young's inequality yields
\begin{equation}\label{I_1_2_Estimate}
\mathrm{\mathbf{I}}_{1;2}^{(k)}\lesssim\varepsilon\left(\Vert \nabla^{k+1} v\Vert
_{L^{2}}^{2}+\Vert
\nabla^{k+2}u\Vert _{L^{2}}^{2} +\Vert \nabla^{k+1}(v+\tau w)\Vert_{L^2}^2\right). 
\end{equation}
Hence, \eqref{I_1_Estimate} holds on account of \eqref{I_1_1_Estimate} and \eqref{I_1_2_Estimate}.  
\end{proof}
\noindent Next we wish to we estimate $\mathrm{\mathbf{I}}_2^{(k)}$, defined in \eqref{I_terms}. 
\begin{lemma}[Estimate of $\mathrm{\mathbf{I}}_2^{(k)}$]\label{I_2_Lemma} For any $1\leq k\leq s$, it holds that 
\begin{equation}\label{I_2_Estimate}
\begin{aligned}
\mathrm{\mathbf{I}}_2^{(k)}\lesssim &\,
\varepsilon\Big(\Vert \nabla^{k+2}u\Vert_{L^2}^2+\Vert \nabla^{k+1} w\Vert_{L^2}^2+ \Vert \nabla^{k+1}(v+\tau w)\Vert_{L^2}^2+\Vert
\nabla^{k+2}v\Vert _{L^{2}}\Big)\\
\lesssim&\,\varepsilon \mathscr{D}^2_k[\mathbf{U}](t).
\end{aligned}
\end{equation}

\end{lemma}
\begin{proof}
Recall that 
\begin{equation}\label{nabla_R_1_k}
\nabla \mathrm{R}^{(k)}=\nabla ^{k+1}\left( \dfrac{B}{A}vw+2\nabla u\nabla
v\right).
\end{equation}%
Thus, we have
\begin{equation}
\begin{aligned}
\mathrm{\mathbf{I}}_2^{(k)}=&\,\int_{\mathbb{R}^{3}}\Big|\nabla ^{k+1}\left( \frac{B}{A}vw+2\nabla u \cdot \nabla
v\right) \nabla^{k+1} (v+\tau w)\Big|\dx\\
=&\, \mathrm{\mathbf{I}}_{2;1}^{(k)}+\mathrm{\mathbf{I}}_{2;2}^{(k)}. 
\end{aligned}
\end{equation}
We estimate $\mathrm{\mathbf{I}}_{2;1}^{(k)}$ as follows: 
\begin{equation}\label{I_2_1_Main}
\begin{aligned}
\mathrm{\mathbf{I}}_{2;1}^{(k)}\lesssim&\,\int_{\mathbb{R}^{3}}\Big|\nabla^{k+1} (vw)\nabla^{k+1} (v+\tau w)\Big|\dx\\
=&\,\int_{\mathbb{R}^{3}}\Big|\sum_{0\leq \ell\leq k+1}\nabla^{k+1-\ell} v\nabla^\ell w\nabla^{k+1}(v+\tau w)\Big|\dx\\
\lesssim&\,\int_{\mathbb{R}^{3}}\Big|\sum_{0\leq \ell\leq k+1}\nabla^{k+1-\ell} v\nabla^\ell w\nabla^{k+1}(v+\tau w)\Big|\dx\\
=&\, C \int_{\mathbb{R}^{3}}\Big|\sum_{0\leq \ell\leq k} \nabla^{k+1-\ell} v\nabla^\ell w\nabla^{k+1}(v+\tau w)  
\Big|\dx\\
&+C\int_{\R^3}\Big|v\nabla^{k+1} w\nabla^{k+1}(v+\tau w)\Big|\dx\\
\lesssim&\,\sum_{0\leq \ell\leq k+1}\Vert \nabla^{k+1-\ell}v\Vert_{L^3}\Vert \nabla^\ell w\Vert_{L^6} \Vert \nabla^{k+1}(v+\tau w)\Vert_{L^2}. 
\end{aligned}
\end{equation}
Using H\"older's inequality, the term on the right-hand side of \eqref{I_2_1_Main} corresponding to $\ell=k+1$ can be estimated in the following manner:
\begin{equation}\label{I_2_1_First_Estimate}
\begin{aligned}
\int_{\R^3}\Big|\nabla^{k+1-\ell} v\nabla^\ell w\nabla^{k+1}(v+\tau w)\Big|\dx\lesssim&\, \Vert v\Vert_{L^\infty}\Vert \nabla^{k+1} w\Vert_{L^2} \Vert \nabla^{k+1}(v+\tau w)\Vert_{L^2}\\
\lesssim&\, \varepsilon \Big(\Vert \nabla^{k+1} w\Vert_{L^2}^2+ \Vert \nabla^{k+1}(v+\tau w)\Vert_{L^2}^2\Big), 
\end{aligned}
\end{equation}
where we have used the Sobolev embedding theorem. 

To estimate the second term,  observe that the term  $\Vert \nabla^\ell w\Vert_{L^6}$ can be handled as in \eqref{Interpo_I_1_1}. In other words,
 \begin{equation}\label{Interpo_I_2_1}
\Vert \nabla^\ell w\Vert_{L^6}\lesssim \Vert
\nabla^{k+1}w\Vert _{L^{2}}^{\frac{1+\ell}{1+k}}\left\Vert w\right\Vert
_{L^{2}}^{1-\frac{1+\ell}{1+k}},\qquad 0\leq \ell\leq k.
\end{equation} To estimate the term $\Vert \nabla^{k+1-\ell}v\Vert_{L^3}$, we apply the Sobolev--Gagliardo--Nirenberg inequality \eqref{Sobolev_Gagl_Ni_Interpolation_ineq_Main},  \begin{equation}\label{v_Inter_I_2_1}
\Vert \nabla^{k+1-\ell}v\Vert_{L^3}\lesssim \left\Vert
\nabla^{m_2+1}v\right\Vert _{L^{2}}^{\frac{1+\ell}{1+k}}\Vert \nabla^{k+2} v\Vert
_{L^{2}}^{1-\frac{1+\ell}{1+k}},\qquad 0\leq \ell\leq k
\end{equation}
with 
\begin{equation}
\frac{1}{3}=\frac{k+1-\ell}{3} + \left( \frac{1}{2}-\frac{k+2}{3}\right)\Big(1-\frac{\ell+1}{k+1}\Big)+\left( \frac{1}{2}-\frac{m_2+1}{3}\right)\frac{\ell+1}{k+1}.
\end{equation}
The above equation gives 
\begin{equation}
m_2=\frac{1+k}{2(1+\ell)}\leq  \frac{1+s}{2}.  
\end{equation}
 As before, for $s_0\geq [(1+s)/2]+1$, we have \[\left\Vert
\nabla^{m_2+1}v(t)\right\Vert _{L^{2}}\lesssim \mathcal{E}_{s_0}[\mathbf{U}](t).\]
Hence, by collecting \eqref{Interpo_I_2_1} and \eqref{v_Inter_I_2_1}, we obtain 
\begin{equation}\label{I_2_1_Second_Estimate}
\begin{aligned}
&\sum_{0\leq \ell\leq k}\Vert \nabla^{k+1-\ell}v\Vert_{L^3}\Vert \nabla^\ell w\Vert_{L^6} \Vert \nabla^{k+1}(v+\tau w)\Vert_{L^2}\\
&\lesssim\sum_{0\leq \ell\leq k}\Vert
\nabla^{k+1}w\Vert _{L^{2}}^{\frac{1+\ell}{1+k}}\left\Vert w\right\Vert
_{L^{2}}^{1-\frac{1+\ell}{1+k}}\left\Vert
\nabla^{m_2+1}v\right\Vert _{L^{2}}^{\frac{1+\ell}{1+k}}\Vert \nabla^{k+2} v\Vert
_{L^{2}}^{1-\frac{1+\ell}{1+k}}\Vert \nabla^{k+1}(v+\tau w)\Vert_{L^2}\\
\lesssim&\, \varepsilon \Big(\Vert
\nabla^{k+1}w\Vert _{L^{2}}^2+\Vert
\nabla^{k+2}v\Vert _{L^{2}}^2+\Vert \nabla^{k+1}(v+\tau w)\Vert_{L^2}^2\Big). 
\end{aligned}
\end{equation}
Hence, collecting \eqref{I_2_1_First_Estimate} and \eqref{I_2_1_Second_Estimate}, we obtain
\begin{equation}
\begin{aligned}\label{I_2_1_Main_estimate}
\mathrm{\mathbf{I}}_{2;1}^{(k)}\lesssim \varepsilon \Big(\Vert \nabla^{k+1} w\Vert_{L^2}^2+ \Vert \nabla^{k+1}(v+\tau w)\Vert_{L^2}^2+\Vert
\nabla^{k+2}v\Vert _{L^{2}}^2\Big).
\end{aligned}
\end{equation}
Next we estimate $\mathrm{\mathbf{I}}_{2;2}^{(k)}$. We have  
 \begin{equation}
 \begin{aligned}
\mathrm{\mathbf{I}}_{2;2}^{(k)}=&2\int_{\mathbb{R}^{3}}\Big|\nabla ^{k+1}(\nabla u\nabla
v) \nabla^{k+1} (v+\tau w)\Big|\dx\\
=&\,C\int_{\mathbb{R}^{3}}\Big|\sum_{0\leq \ell\leq k+1}\nabla^{k+2-\ell} u\nabla^{\ell+1} v\nabla^{k+1}(v+\tau w)\Big|\dx.
\end{aligned}
\end{equation}
We split  the above sum into three cases: $\ell=0$, $\ell=k+1$, and $1\leq\ell\leq k$. Thus
\begin{equation}\label{I_2_2_Estimate_1}
\begin{aligned}
\mathrm{\mathbf{I}}_{2;2}^{(k)}\lesssim&\,\int_{\mathbb{R}^{3}}\Big|\nabla^{k+2} u\nabla v\nabla^{k+1}(v+\tau w)\Big|\dx\\
&+\int_{\mathbb{R}^{3}}\Big|\nabla u\nabla^{k+2} v\nabla^{k+1}(v+\tau w)\Big|\dx \\
&+\int_{\mathbb{R}^{3}}\Big|\sum_{1\leq \ell\leq k}\nabla^{k+2-\ell} u\nabla^{\ell+1} v\nabla^{k+1}(v+\tau w)\Big|\dx. 
\end{aligned}
\end{equation}
As before, the first term in \eqref{I_2_2_Estimate_1} is estimated by using H\"older's inequality and the Sobolev embedding theorem, 
 \begin{equation}\label{Estimate_l_0}
\begin{aligned}
\int_{\mathbb{R}^{3}}\Big|\nabla^{k+2} u\nabla v\nabla^{k+1}(v+\tau w)\Big|\dx\lesssim &\,\Vert \nabla v\Vert_{L^\infty}\Vert \nabla^{k+2}u\Vert_{L^2}\Vert\nabla^{k+1}(v+\tau w) \Vert_{L^2}\\
\lesssim&\, \varepsilon \Big(\Vert \nabla^{k+2}u\Vert_{L^2}^2+\Vert\nabla^{k+1}(v+\tau w) \Vert_{L^2}^2\Big). 
\end{aligned}
\end{equation}
Similarly, we estimate the second term on the right-hand side of \eqref{I_2_2_Estimate_1} as 
 \begin{equation}\label{Estimate_l_k+1}
 \begin{aligned}  
\int_{\mathbb{R}^{3}}\Big|\nabla u\nabla^{k+2} v\nabla^{k+1}(v+\tau w)\Big|\dx \lesssim&\, \Vert \nabla u\Vert_{L^\infty}\Vert \nabla^{k+2} v\Vert_{L^2} \Vert \nabla^{k+1}(v+\tau w)\Vert_{L^2}\\
\lesssim&\, \varepsilon \Big(\Vert \nabla^{k+2} v\Vert_{L^2}^2+ \Vert \nabla^{k+1}(v+\tau w)\Vert_{L^2}^2\Big).
\end{aligned}
\end{equation}
For the third  term on the right-hand side of \eqref{I_2_2_Estimate_1}, we write 
\begin{equation}
\begin{aligned}
&\int_{\mathbb{R}^{3}}\Big|\sum_{1\leq \ell\leq k}\nabla^{k+2-\ell} u\nabla^{\ell+1} v\nabla^{k+1}(v+\tau w)\Big|\dx. \\
\lesssim&\, \sum_{1\leq \ell\leq k}\Vert \nabla^{k+2-\ell}u\Vert_{L^3}\Vert \nabla^{\ell+1} v\Vert_{L^6} \Vert \nabla^{k+1}(v+\tau w)\Vert_{L^2}.
\end{aligned}
\end{equation}
By applying the Gagliardo--Nirenberg inequality \eqref{Interpolation_inequality}, we obtain 
\begin{equation}
\Vert \nabla^{\ell+1} v\Vert_{L^6}\lesssim \Vert  v\Vert_{L^\infty}^{1-\frac{2\ell+1}{2k+1}}\Vert  \nabla^{k+2}v\Vert_{L^2}^{\frac{2\ell+1}{2k+1}},\qquad 1\leq \ell\leq k.
\end{equation}  
We also have, by suing \eqref{Sobolev_Gagl_Ni_Interpolation_ineq_Main}, 
\begin{equation}
\Vert \nabla^{k+2-\ell}u\Vert_{L^3}\lesssim \left\Vert
\nabla^{m_3+1}u\right\Vert _{L^{2}}^{\frac{2\ell+1}{2k+1}}\Vert \nabla^{k+2} u\Vert
_{L^{2}}^{1-\frac{2\ell+1}{2k+1}},\qquad 1\leq \ell\leq k
\end{equation}
with 
\begin{equation}
\frac{1}{3}=\frac{k+2-\ell}{3} + \left( \frac{1}{2}-\frac{k+2}{3}\right)\Big(1-\frac{2\ell+1}{2k+1}\Big)+\left( \frac{1}{2}-\frac{m_3+1}{3}\right)\frac{2\ell+1}{2k+1}.
\end{equation}
This yields 
\begin{equation}
m_3=\frac{1}{2}+\frac{1+2k}{1+2l}\leq \frac{2k}{3}+\frac{5}{6}\leq \frac{2s}{3}+\frac{5}{6},\quad \text{since}\quad \ell\geq 1.  
\end{equation}
Hence, for $s_0\geq [2s/3]+1$, we have $\left\Vert
\nabla^{m_3+1}u(t)\right\Vert _{L^{2}}\lesssim \mathcal{E}_{s_0}(t)$. Also, using the Sobolev embedding theorem together with  \eqref{boot_strap_Assum}, we obtain 
$\Vert  v\Vert_{L^\infty}\lesssim \mathcal{E}_{s_0}[\mathbf{U}](t)\lesssim \varepsilon $. Consequently, we obtain  
\begin{equation}\label{Term_l_1}
\begin{aligned}
&\int_{\mathbb{R}^{3}}\Big|\sum_{1\leq \ell\leq k}\nabla^{k+2-\ell} u\nabla^{\ell+1} v\nabla^{k+1}(v+\tau w)\Big|\dx. \\
\lesssim&\sum_{1\leq \ell\leq k}\left\Vert
\nabla^{m_3+1}u\right\Vert _{L^{2}}^{\frac{2\ell+1}{2k+1}}\Vert  v\Vert_{L^\infty}^{1-\frac{2\ell+1}{2k+1}}\Vert \nabla^{k+2} u\Vert
_{L^{2}}^{1-\frac{2\ell+1}{2k+1}}\Vert  \nabla^{k+2}v\Vert_{L^2}^{\frac{2\ell+1}{2k+1}} \Vert \nabla^{k+1}(v+\tau w)\Vert_{L^2}\\
\lesssim&\, \varepsilon \Big(\Vert \nabla^{k+2} u\Vert
_{L^{2}}^2+\Vert  \nabla^{k+2}v\Vert_{L^2}^2+\Vert \nabla^{k+1}(v+\tau w)\Vert_{L^2}^2\Big). 
\end{aligned}
\end{equation}
Therefore, from \eqref{Estimate_l_0}, \eqref{Estimate_l_k+1} and \eqref{Term_l_1}, we  deduce that 
\begin{equation}\label{I_2_2_Main_Estimate}
\begin{aligned}
\mathrm{\mathbf{I}}_{2;2}^{(k)}\lesssim&\,\varepsilon \Big(\Vert \nabla^{k+2}u\Vert_{L^2}^2+\Vert \nabla^{k+2} v\Vert_{L^2}^2+\Vert\nabla^{k+1}(v+\tau w) \Vert_{L^2}^2\Big). 
\end{aligned}
\end{equation}
Hence, \eqref{I_2_Estimate} holds by collecting \eqref{I_2_1_Main_estimate} and \eqref{I_2_2_Main_Estimate}. This finishes the proof of Lemma \ref{I_2_Lemma}. 
\end{proof}

The estimate of $\mathrm{\mathbf{I}}_4^{(k)}$ can be done as the one of $\mathrm{\mathbf{I}}_2^{(k)}$, we thus omit the details and just state the result. 
\begin{lemma}[Estimate of $\mathrm{\mathbf{I}}_4^{(k)}$]\label{Lemma_I_4}
For any $1\leq k\leq s$, it holds that 
\begin{equation}\label{I_4_Estimate}
\begin{aligned}
\mathrm{\mathbf{I}}_4^{(k)}\lesssim &\,\varepsilon \Big(\Vert \nabla^{k+2}u\Vert_{L^2}^2+\Vert \nabla^{k+1} w\Vert_{L^2}^2+ \Vert \nabla^{k+1}v\Vert_{L^2}^2+\Vert
\nabla^{k+2}v\Vert _{L^{2}}\Big)\\
\lesssim&\, \varepsilon  \mathscr{D}_k^2[\mathbf{U}](t). 
\end{aligned}
\end{equation}
\end{lemma}

Our goal now is to estimate $\mathrm{\mathbf{I}}_3^{(k)}$.  
\begin{lemma}[Estimate of $\mathrm{\mathbf{I}}_3^{(k)}$]\label{Lemma_I_3} 
For any $1\leq k\leq s$, it holds that 
\begin{equation}\label{I_3_Estimate}
\begin{aligned}
\mathrm{\mathbf{I}}_3^{(k)}\lesssim &\,\varepsilon\Big(\Vert \nabla^{k+1} v\Vert
_{L^{2}}^{2}+\Vert
\nabla^{k+2}u\Vert _{L^{2}}^{2} +\Vert
\nabla^{k+1}u\Vert _{L^{2}}^{2}\Big.\\
\Big.&+\Vert
\nabla^{k+1}w\Vert _{L^{2}}^{2}+\Vert \Delta\nabla^{k}(u+\tau v)\Vert_{L^2}^2\Big)\\
\lesssim& \,\varepsilon\big( \mathscr{D}^2_{k-1}[\mathbf{U}](t)+\mathscr{D}^2_k[\mathbf{U}](t)\big). 
\end{aligned}
\end{equation}
\end{lemma}
\begin{proof}

We have 
\begin{equation}
\begin{aligned}
\mathrm{\mathbf{I}}_3^{(k)}=&\,\int_{\mathbb{R}^{3}}|\mathrm{R}^{(k)}\Delta \left( U+\tau
V\right)| \dx\\
\lesssim &\,\int_{\mathbb{R}^{3}}\Big|\sum_{0\leq \ell\leq k} \nabla^{k-\ell} v\nabla^\ell w\Delta\nabla^{k}(u+\tau v)\Big|\dx\\
&+\int_{\mathbb{R}^{3}}\Big|\sum_{0\leq \ell\leq k}\nabla^{k-\ell} \nabla u\nabla^\ell\nabla  v\Delta \nabla^{k}(u+\tau v)\Big|\dx\\
=&\, \mathrm{\mathbf{I}}_{3;1}^{(k)}+\mathrm{\mathbf{I}}_{3;2}^{(k)}. 
\end{aligned}
\end{equation}

First, we estimate $\mathrm{\mathbf{I}}_{3;1}^{(k)}$. We have 
\begin{equation}\label{I_3_1_1}
\begin{aligned}
\mathrm{\mathbf{I}}_{3;1}^{(k)}\lesssim \sum_{0\leq \ell\leq k}\Vert \nabla^{k-\ell}v\Vert_{L^3}\Vert \nabla^\ell w\Vert_{L^6} \Vert\Delta \nabla^{k}(u+\tau v)\Vert_{L^2}.
\end{aligned}
\end{equation}
Using \eqref{Interpolation_inequality},  we write 
 \begin{equation}\label{Interpo_I_1_3}
\Vert \nabla^\ell w\Vert_{L^6}\lesssim \Vert
\nabla^{k+1}w\Vert _{L^{2}}^{\frac{1+\ell}{1+k}}\left\Vert w\right\Vert
_{L^{2}}^{1-\frac{1+\ell}{1+k}},\qquad 0\leq \ell\leq k.
\end{equation}
Applying \eqref{Sobolev_Gagl_Ni_Interpolation_ineq_Main}, we obtain 
 \begin{equation}\label{Interpo_I_3_2}
\Vert\nabla^{k-\ell} v\Vert_{L^3}\lesssim \left\Vert
\nabla^{m_4}v\right\Vert _{L^{2}}^{\frac{\ell+1}{k+1}}\Vert \nabla^{k+1} v\Vert
_{L^{2}}^{1-\frac{\ell+1}{k+1}},\qquad 0\leq \ell\leq k
\end{equation}
with 
\begin{equation}
\frac{1}{3}=\frac{k-\ell}{3}+\left( \frac{1}{2}-\frac{m_4}{3}\right)\frac{\ell+1}{k+1} + \left( \frac{1}{2}-\frac{k+1}{3}\right)\Big(1-\frac{\ell+1}{k+1}\Big).
\end{equation}
This results in 
\begin{equation}
m_4=\frac{k+1}{2(1+\ell)}\leq \frac{k+1}{2}\leq \frac{s+1}{2}. 
\end{equation}
Therefore, for $s_0\geq [s/2]+1$, we have $\left\Vert
\nabla^{m_4}v(t)\right\Vert _{L^{2}}\lesssim \mathcal{E}_{s_0}[\mathbf{U}](t) $. 
Hence,  inserting \eqref{Interpo_I_1_3} and \eqref{Interpo_I_3_2} into \eqref{I_3_1_1}, we obtain, by making use of \eqref{boot_strap_Assum},  
\begin{equation}\label{I_3_1_2}
\begin{aligned}
\mathrm{\mathbf{I}}_{3;1}^{(k)}\lesssim &\,\sum_{0\leq \ell\leq k}\left\Vert
\nabla^{m_4}v\right\Vert _{L^{2}}^{\frac{\ell+1}{k+1}}\Vert \nabla^{k+1} v\Vert
_{L^{2}}^{1-\frac{\ell+1}{k+1}}\Vert
\nabla^{k+1}w\Vert _{L^{2}}^{\frac{1+\ell}{1+k}}\left\Vert w\right\Vert
_{L^{2}}^{1-\frac{1+\ell}{1+k}} \Vert\Delta \nabla^{k}(v+\tau w)\Vert_{L^2}\\
\lesssim&\, \varepsilon\Big(\Vert \nabla^{k+1} v\Vert
_{L^{2}}^{2}+\Vert
\nabla^{k+1}w\Vert _{L^{2}}^{2} +\Vert \Delta\nabla^{k}(u+\tau v)\Vert_{L^2}^2\Big). 
\end{aligned}
\end{equation}
Next, we estimate $\mathrm{\mathbf{I}}_{3;2}^{(k)}$. Recall that 
\begin{equation}\label{I_3_2_Estimate_1}
\begin{aligned}
\mathrm{\mathbf{I}}_{3;2}^{(k)}=&\,C\int_{\mathbb{R}^{3}}\Big|\sum_{0\leq \ell\leq k}\nabla^{k-\ell} \nabla u\nabla^\ell\nabla  v\Delta \nabla^{k}(u+\tau v)\Big|\dx\\
\lesssim &\,\int_{\mathbb{R}^{3}}\Big|\nabla^{k} \nabla u\nabla  v\Delta \nabla^{k}(u+\tau v)\Big|\dx\\
&+\int_{\mathbb{R}^{3}}\Big|  \nabla u\nabla^k\nabla  v\Delta \nabla^{k}(u+\tau v)\Big|\dx\\
&+\int_{\mathbb{R}^{3}}\Big|\sum_{1\leq \ell\leq k-1}\nabla^{k-\ell} \nabla u\nabla^\ell\nabla  v\Delta \nabla^{k}(u+\tau v)\Big|\dx.
\end{aligned}
\end{equation}
We estimate the first term on the right-hand side of \eqref{I_3_2_Estimate_1} as
\begin{equation}\label{First_Term_I_3_2}
\begin{aligned}
\int_{\mathbb{R}^{3}}\Big|\nabla^{k} \nabla u\nabla  v\Delta \nabla^{k}(u+\tau v)\Big|\dx\lesssim&\, \Vert \nabla v\Vert_{L^\infty}\Vert \nabla^{k+1}u\Vert_{L^2}\Vert\Delta \nabla^{k}(u+\tau v) \Vert_{L^2}\\
\lesssim&\,\varepsilon\Big(\Vert
\nabla^{k+1}u\Vert _{L^{2}}^{2} +\Vert \Delta\nabla^{k}(u+\tau v)\Vert_{L^2}^2\Big). 
\end{aligned}
\end{equation}
The second term on the right-hand side of \eqref{I_3_2_Estimate_1} is estimated as 
\begin{equation}\label{Second_Term_I_3_2}
\begin{aligned}
\int_{\mathbb{R}^{3}}\Big|  \nabla u\nabla^k\nabla  v\Delta \nabla^{k}(u+\tau v)\Big|\dx\lesssim &\, \Vert \nabla u\Vert_{L^\infty}\Vert \nabla^{k+1}v\Vert_{L^2}\Vert\Delta \nabla^{k}(u+\tau v) \Vert_{L^2}\\
\lesssim &\,\varepsilon\Big(\Vert
\nabla^{k+1}v\Vert _{L^{2}}^{2} +\Vert \Delta\nabla^{k}(u+\tau v)\Vert_{L^2}^2\Big).
\end{aligned}
\end{equation}

For the last term on the right-hand side of \eqref{I_3_2_Estimate_1}, we have 
\begin{equation}\label{I_3_2_Estimate_1}
 \begin{aligned}
&\int_{\mathbb{R}^{3}}\Big|\sum_{1\leq \ell\leq k-1}\nabla^{k-\ell} \nabla u\nabla^\ell\nabla  v\Delta \nabla^{k}(u+\tau v)\Big|\dx\\
&+\sum_{1\leq \ell\leq k-1}\Vert \nabla^{k+1-\ell}u\Vert_{L^3}\Vert \nabla^{\ell+1} v\Vert_{L^6} \Vert\Delta \nabla^{k}(u+\tau v)\Vert_{L^2}. 
\end{aligned}
\end{equation}

 We have by exploiting \eqref{Interpolation_inequality},
 \begin{equation}
\Vert \nabla^{\ell+1} v\Vert_{L^6}\lesssim \Vert  v\Vert_{L^2}^{1-\frac{2+\ell}{1+k}}\Vert  \nabla^{k+1}v\Vert_{L^2}^{\frac{2+\ell}{1+k}},\qquad 1\leq \ell\leq k-1. 
\end{equation} 
As before, we apply \eqref{Sobolev_Gagl_Ni_Interpolation_ineq_Main} and estimate $\Vert \nabla^{k+1-\ell}u\Vert_{L^3}$ as follows: 
\begin{equation}
\Vert \nabla^{k+1-\ell}u\Vert_{L^3}\lesssim \left\Vert
\nabla^{m_5+1}u\right\Vert _{L^{2}}^{\frac{2+\ell}{1+k}}\Vert \nabla^{k+2} u\Vert
_{L^{2}}^{1-\frac{2+\ell}{1+k}},\qquad 1\leq \ell\leq k-1,
\end{equation}
where \begin{equation}
\frac{1}{3}=\frac{k+1-\ell}{3} + \left( \frac{1}{2}-\frac{k+2}{3}\right)\Big(1-\frac{2+\ell}{1+k}\Big)+\left( \frac{1}{2}-\frac{m_5+1}{3}\right)\frac{2+\ell}{1+k},
\end{equation}
which implies 
\begin{equation}
m_5=\frac{3(1+k)}{2(2+\ell)}\leq \frac{k+1}{2}\leq \frac{s+1}{2},\quad \text{since}\quad \ell\geq 1.  
\end{equation}
Hence, as before, this implies that  for $s_0\geq [s/2]+1$, we have $\left\Vert
\nabla^{m_5+1}u(t)\right\Vert _{L^{2}}\lesssim \mathcal{E}_{s_0}[\mathbf{U}](t) $. Consequently,  we obtain from above 
\begin{equation}\label{Third_Term_I_3_1}
\begin{aligned}
&\sum_{1\leq \ell\leq k-1}\Vert \nabla^{k+1-\ell}u\Vert_{L^3}\Vert \nabla^{\ell+1} v\Vert_{L^6} \Vert\Delta \nabla^{k}(u+\tau v)\Vert_{L^2}\\
\lesssim&\, \sum_{1\leq \ell\leq k-1}\left\Vert
\nabla^{m_5+1}u\right\Vert _{L^{2}}^{\frac{2+\ell}{1+k}}\Vert \nabla^{k+2} u\Vert
_{L^{2}}^{1-\frac{2+\ell}{1+k}}\Vert  v\Vert_{L^2}^{1-\frac{2+\ell}{1+k}}\Vert  \nabla^{k+1}v\Vert_{L^2}^{\frac{2+\ell}{1+k}}\Vert\Delta \nabla^{k}(u+\tau v)\Vert_{L^2}\\
\lesssim&\,\varepsilon\Big(\Vert \nabla^{k+1} v\Vert
_{L^{2}}^{2}+\Vert
\nabla^{k+2}u\Vert _{L^{2}}^{2} +\Vert \Delta\nabla^{k}(u+\tau v)\Vert_{L^2}^2\Big).  
\end{aligned}
\end{equation}
Therefore,  from \eqref{First_Term_I_3_2}, \eqref{Second_Term_I_3_2} and \eqref{Third_Term_I_3_1}, we deduce that 
\begin{equation}\label{I_3_2_Main_Estimate}
\begin{aligned}
\mathrm{\mathbf{I}}_{3;2}^{(k)}\lesssim \varepsilon\Big(\Vert \nabla^{k+1} v\Vert
_{L^{2}}^{2}+\Vert
\nabla^{k+2}u\Vert _{L^{2}}^{2} +\Vert
\nabla^{k+1}u\Vert _{L^{2}}^{2}+\Vert \Delta\nabla^{k}(u+\tau v)\Vert_{L^2}^2\Big).
\end{aligned}
\end{equation}
Putting together \eqref{I_3_1_2} and \eqref{I_3_2_Main_Estimate} yields \eqref{I_3_Estimate}. 
\end{proof}
\noindent Next we derive a bound for $\mathrm{\mathbf{I}}_5^{(k)}$. 
\begin{lemma}[Estimate of $\mathrm{\mathbf{I}}_5^{(k)}$]\label{Lemma_I_5}
For any $1\leq k\leq s$, it holds that 
\begin{equation}\label{I_5_Estimate}
\begin{aligned}
\mathrm{\mathbf{I}}_5^{(k)}\lesssim &\,\varepsilon\Big(\Vert \nabla^{k+1} v\Vert
_{L^{2}}^{2}+\Vert
\nabla^{k+2}u\Vert _{L^{2}}^{2} +\Vert
\nabla^{k+1}u\Vert _{L^{2}}^{2}\Big.\\
\Big.&+\Vert
\nabla^{k+1}w\Vert _{L^{2}}^{2}+\Vert \Delta\nabla^{k}(u+\tau v)\Vert_{L^2}^2\Big)\\
\lesssim& \,\varepsilon\big( \mathscr{D}^2_{k-1}[\mathbf{U}](t)+\mathscr{D}^2_k[\mathbf{U}](t)\big). 
\end{aligned}
\end{equation}
\end{lemma}
\begin{proof}
The proof of Lemma \ref{Lemma_I_5} can be done as the one of Lemma \ref{Lemma_I_3}, where $\Vert \Delta \nabla^k (u+\tau v)\Vert_{L^2}$ is replaced by $\Vert \nabla^{k} w\Vert_{L^2}$. We omit the details here.   
\end{proof}


 \subsection{Proof of Theorem \ref{Main_Theorem}}\label{Section_Proof_Theorem_1}
 
  Let $s\geq 3$ and $s_0=\max\{[2s/3]+1,[s/2]+2\}\leq m\leq s $. By plugging the estimates \eqref{I_1_Estimate}, \eqref{I_2_Estimate}, \eqref{I_4_Estimate}, \eqref{I_3_Estimate} and \eqref{I_5_Estimate} into  \eqref{E_I_Est},   and keeping in mind \eqref{Dissipative_weighted_norm_1}, we obtain 

\begin{equation}  \label{E_I_Est_k_1}
\begin{aligned}
\mathcal{E}^2_k[\mathbf{U}](t)+\mathcal{D}_k^2(t)
 \lesssim\, \mathcal{E}^2_k[\mathbf{U}](0)+\varepsilon \left(\mathcal{D}_k^2[\mathbf{U}](t)+\mathcal{D}_{k-1}^2(t)[\mathbf{U}]\right),\qquad 1\leq k\leq s.  
\end{aligned}
\end{equation}
Summing the above estimate over $k$ from $k=1$ to $k=s_0$ and adding the result to \eqref{Main_Estimate_D_0}, we obtain 
\begin{equation}  \label{E_I_Est_k_1}
\mathcal{E}^2_{s_0}[\mathbf{U}](t)+\mathcal{D}_{s_0}^2[\mathbf{U}](t)
\leq \mathcal{E}^2_{s_0}[\mathbf{U}](0)+\varepsilon \mathcal{D}_{s_0}^2[\mathbf{U}](t).  
\end{equation}
For $\varepsilon>0$ sufficiently small, this yields 
 \begin{equation}
\mathcal{E}^2_{s_0}[\mathbf{U}](t)+\mathcal{D}_{s_0}^2[\mathbf{U}](t)
\leq \mathcal{E}^2_{s_0}[\mathbf{U}](0). 
\end{equation}
By assuming (as in \eqref{Initial_Assumption_Samll}) the initial energy satisfies 
\[\mathcal{E}^2_{s_0}[\mathbf{U}](0)\leq \delta< \frac{\varepsilon^2}{2},\]  we obtain \[\mathcal{E}^2_{s_0}[\mathbf{U}](t)\leq \frac{\varepsilon^2}{2},\] which closes the a priori estimate \eqref{boot_strap_Assum} by a standard continuity argument.  

Now, by summing \eqref{E_I_Est_k_1} over $1\leq k\leq m$, adding the result to \eqref{Main_Estimate_D_0}, and selecting $\varepsilon>0$ small enough, we obtain 
\begin{equation}
\mathcal{E}^2_{m}[\mathbf{U}](t)+\mathcal{D}_{m}^2[\mathbf{U}](t)
\leq \mathcal{E}^2_{m}[\mathbf{U}](0), \qquad t \geq 0,
\end{equation}
which is exactly \eqref{Main_Energy_Estimate}. This finishes the proof of Theorem \ref{Main_Theorem}.

  \section{The decay estimates--Proof of Theorem \ref{Theorem_Decay}}\label{Sec: Decay_Linearized}
  Our main goal in this section is to prove Theorem \ref{Theorem_Decay}.  
  
  We consider the linearized problem:
  \begin{subequations}\label{Main_System_First_Order_Linear}
\begin{equation}
\left\{
\begin{array}{ll}
u_{t}=v,\vspace{0.1cm} &  \\
v_{t}=w,\vspace{0.1cm} &  \\
\tau w_{t}=\Delta u+\beta \Delta v-w, &
\end{array}%
\right.  \label{System_New_Linear}
\end{equation}
with the initial data

\begin{eqnarray}  \label{Initial_Condition_Linear}
u(t=0)=u_0,\qquad v(t=0)=v_0,\qquad w(t=0)=w_0.
\end{eqnarray}
\end{subequations}

  Now,   we derive an energy estimate for the negative Sobolev norm of the solution of \eqref{Main_System_First_Order_Linear}. 
We apply $\Lambda^{-\gamma}$ to \eqref{System_New_Linear} and set $\tilde{u}=\Lambda^{-\gamma}u$, $\tilde{v}=\Lambda^{-\gamma}v$, and $\tilde{w}=\Lambda^{-\gamma}w$. This yields 
\begin{equation}
\left\{
\begin{array}{ll}
\tilde{u}_{t}=\tilde{v},\vspace{0.1cm} &  \\
\tilde{v}_{t}=\tilde{w},\vspace{0.1cm} &  \\
\tau \tilde{w}_{t}=\Delta \tilde{u}+\beta \Delta \tilde{v}-\tilde{w}, &
\end{array}%
\right.  \label{System_New_Lambda}
\end{equation}
We have the following Proposition.  
\begin{proposition}\label{Proposition_gamma}
Let $\gamma>0$, then it holds that 
\begin{equation}\label{Energy_Estimate_gamma_1}
\mathcal{E}^2_{-\gamma}[\mathbf{U}](t)+\mathcal{D}_{-\gamma}^2[\mathbf{U}](t)
\leq \mathcal{E}^2_{-\gamma}[\mathbf{U}](0). 
\end{equation}
\end{proposition}
Following similar reasoning as before, and using system \eqref{System_New_Lambda},  we obtain \eqref{Energy_Estimate_gamma_1}. We omit the details. 

 Our next goal is to prove the decay bound \eqref{Decay_1}. We point out that we cannot apply directly the method in \cite{Guo_Wang_2012} to get  the decay estimates due to the restricted use of the interpolation inequality in Soblev spaces with negative index:
\begin{equation}\label{Fractional_Gag_Nirenberg}
\Vert \nabla ^{\ell}f\Vert _{L^{2}}\leq C\Vert
\nabla^{\ell+1}f\Vert _{L^{2}}^{1-\theta}\Vert \Lambda^{-\gamma} f\Vert
_{L^{2}}^{\theta},   \qquad \text{where}\qquad \theta=\frac{1}{\ell+\gamma+1}; 
\end{equation}
cf. Lemma~\ref{Lemma_gamma_Interpo}. To overcome this difficulty and inspired by \cite{Xu_Kawashima_2015}, the strategy 
  is to split the solution into a low-frequency and a high-frequency part instead. 
  
  Hence, let us consider the unit decomposition 
  \begin{equation}    
1=\Psi(\xi)+\Phi(\xi)
\end{equation}
where $\Psi,\,\Phi\in C_c^\infty (\R^3)$, $0\leq \Psi(\xi),\,\Phi(\xi)\leq 1$ satisfy 
\begin{equation}
\begin{aligned}
\Psi(\xi)=1, \quad \text{if}\quad |\xi|\leq \mathrm{R},\quad \Psi(\xi)=0, \quad \text{if}\quad |\xi|\geq 2R
\end{aligned}
\end{equation}
with $R>0$.
 We define $\mathbf{L}_R$ and $\mathbf{H}_R$  as follows: 
\begin{equation}
\widehat{\mathbf{L}_R f}(\xi)=\Psi(\xi) \hat{f}(\xi)\qquad \text{and}\qquad \widehat{\mathbf{H}_R f}(\xi)=\Phi(\xi) \hat{f}(\xi). 
\end{equation}
Accordingly, 
\begin{equation}\label{Cut-off_Operator}
f^{\mathrm{L}}=\mathbf{L}_R f\qquad \text{and}\qquad f^{\mathrm {H}}=\mathbf{H}_Rf. 
\end{equation}

We denote by  $(\hat{u}, \hat{v}, \hat{w})(\xi,t)$ the Fourier transform of the solution of \eqref{System_New_Linear}.  That is,  $(\hat{u}, \hat{v}, \hat{w})(\xi,t)=\mathscr{F}[(u,v,w)(x,t)]$. We define 
\begin{equation}\label{Energy_Fourier}
\begin{aligned}
\hat{E} (\xi,t)=&\,\frac{1}{2}\left\{|\hat{v}+\tau \hat{w}|^2+\tau (\beta-\tau )|\xi|^2|\hat{v}|^2+|\xi|^2|\hat{u}+\tau \hat{v}|^2\right\}\\
=&\,\frac{1}{2}|\hat{V}(\xi,t)|^2
\end{aligned}
\end{equation}
with $V=(v +\tau w , \nabla(u + \tau v),\nabla v)$.

We have the following lemma. 
\begin{lemma}
Assume that  $0<\tau<\beta$. Then,  there exists a Lyapunov functional $\hat{L}(\xi, t)$ satisfying for all $t\geq 0$
 \begin{equation}\label{Equiv_E_L_Linear}
\hat{L}(\xi, t)\approx  \hat{E} (\xi,t)\approx |\hat{V}(\xi,t)|^2
\end{equation}
and 
 \begin{equation}\label{Lyapunov_main_Linear}
\frac{\textup {d}}{\textup {d}t} \hat{L}(\xi,t)+c\frac{|\xi|^2}{1+|\xi|^2}\hat{E}(\xi,t)\leq 0. 
\end{equation}
\end{lemma}  

The functional $\hat{L}(\xi,t)$ is the same one defined in \cite[Eq. (3.20)]{PellSaid_2019_1}.  The proof of \eqref{Equiv_E_L_Linear} was given in \cite[(2.23)]{PellSaid_2019_1}, while the proof of \eqref{Lyapunov_main_Linear} was  in \cite[(3.22)]{PellSaid_2019_1}.

\subsection{Proof of the estimate  \eqref{Decay_1}}
In this section, we prove the decay estimate \eqref{Decay_1}.  
We consider system \eqref{Main_System_First_Order_Linear},
 write $\mathbf{U}=\mathbf{U}^{\mathrm{L}}+\mathbf{U}^\mathrm{H}$ 
with  $\mathbf{U}=(u, v, w) $  is the solution of 
\eqref{Main_System_First_Order_Linear},  $\mathbf{U}^\mathrm{L}=(u^\mathrm{L}, v^\mathrm{L}, w^\mathrm{L})$ and $\mathbf{U}^\mathrm{H}=(u^{\mathrm{H}}, v^{\mathrm{H}}, w^{\mathrm{H}})$ (see \cite{Xu_Kawashima_2015} for similar ideas)  

\begin{description}
\item[Case 1](high frequency) 
\end{description}

We multiply the inequality \eqref{Lyapunov_main_Linear} by $\Phi^2$, we get 
\begin{equation}
\frac{\textup {d}}{\textup {d}t}\big(\Phi^2 \hat{L}(\xi,t)\big)+c\frac{R^2}{1+R^2}\big(\Phi^2\hat{E}(\xi,t)\big)\leq 0.
\end{equation}
This implies by using \eqref{Equiv_E_L_Linear} together with \eqref{Energy_Fourier} and Plancherel's identity   
\begin{equation}\label{Decay_High_Fre}
\Vert V^{\mathrm{L}}(t)\Vert_{L^2}\lesssim \Vert V_0\Vert_{L^2}e^{-c_2t}, 
\end{equation}
where the constant $c_2>0$ depends on $R$.  

\begin{description}
\item[Case 2](low frequency) 
\end{description}

Now multiplying \eqref{Lyapunov_main_Linear} by  
$\Psi^2$, we get  
\begin{equation}
\frac{\textup {d}}{\textup {d}t} \big(\Psi^2\hat{L}(\xi,t)\big)+c\frac{|\xi|^2}{1+R^2}\hat{E}(\xi,t)\leq 0.
\end{equation}
Hence, using Plancherel's identity as above, we get
\begin{equation}\label{Lyap_Phys}
\frac{\textup {d}}{\textup {d}t}\mathcal{L^\mathrm{L}}(t)+c_3 \Vert \nabla V^{\mathrm{H}}(t)\Vert_{L^2}^2\leq 0,
\end{equation}
  where 
 \begin{equation}
\mathcal{L^\mathrm{L}}(t)=\int_{\R^3_\xi} \Psi^2(\xi)L(\xi,t)\textup{d}\xi, 
\end{equation}
 and the constant $c_3>0$ depends on $R$. 
 

Applying Lemma \ref{Lemma_gamma_Interpo}, we have 
\begin{equation}\label{Main_Inter_Inequality}
\Vert  V\Vert _{L^{2}}^{1+\frac{1}{\gamma}}\Vert \Lambda^{-\gamma} V\Vert
_{L^{2}}^{-\frac{1}{\gamma}}\lesssim \Vert
\nabla V\Vert _{L^{2}}.
\end{equation}
Using the fact that 
\begin{equation}\label{V_gamma_Norm}
\Vert  \Lambda^{-\gamma} V(t)\Vert_{L^2}\lesssim \mathcal{E}_{-\gamma}[\mathbf{U}](0), 
\end{equation}
together with \eqref{Main_Inter_Inequality}, we obtain from \eqref{Lyap_Phys}, that 
\begin{equation}\label{Lyapunov_2}
\frac{\textup {d}}{\textup {d}t}\mathcal{L^\mathrm{L}}(t)+C \Vert  V^{\mathrm{L}}\Vert _{L^{2}}^{2(1+1/\gamma)}\Big(\mathcal{E}_{-\gamma}[\mathbf{U}](0)\Big)^{-\frac{2}{\gamma}}\leq 0,
\end{equation}
where we have used the fact that $\mathcal{E}_{-\gamma}[\mathbf{U}^{\mathrm{L}}](0)\leq \mathcal{E}_{-\gamma}[\mathbf{U}](0).$

It is clear that   
\begin{equation}\label{Equiv_L_L_V}
\mathcal{L^\mathrm{L}}(t)\approx \Vert  V^{\mathrm{L}}\Vert _{L^{2}}^2,\qquad \forall t\geq 0. 
\end{equation}
 Hence, we get from \eqref{Lyapunov_2}, 
\begin{equation}\label{Lyapunov_2}
\frac{\textup {d}}{\textup {d}t}\mathcal{L^\mathrm{L}}(t)+C \big(\mathcal{L^\mathrm{L}}(t)\big)^{1+1/\gamma}\Big(\mathcal{E}_{-\gamma}[\mathbf{U}](0)\Big)^{-\frac{2}{\gamma}}\leq 0,
\end{equation}
Integrating this last inequality, we obtain 
\begin{equation}
\mathcal{L^\mathrm{L}}(t)\leq  C_0(1+t)^{-\gamma}
\end{equation}
where $C_0$ is a positive constant depending on $\mathcal{E}_{-\gamma}[\mathbf{U}](0)$. Using \eqref{Equiv_L_L_V} once again, we  obtain 
\begin{equation}\label{Decay_Low_Fre}
\Vert V^{\mathrm{L}}(t)\Vert _{L^{2}}\leq C_0 (1+t)^{-\gamma/2}. 
\end{equation}
Collecting \eqref{Decay_High_Fre} and \eqref{Decay_Low_Fre}, we obtain our decay estimate \eqref{Decay_1}.

\appendix
\begin{appendices}   
\section{ Auxiliary inequalities}\label{Appendix_Usefull_Ineq}

 In this appendix, we recall some inequalities that have been frequently used in the preceding sections.

\begin{lemma}[The Gagliardo--Nirenberg interpolation inequality; See~\cite{Ner59}] \label{interpolation_lemma}  Let $1\leq
p,\,q\,,r\leq \infty $, and let $m$ be a positive integer. Then for any
integer $j$ with $0\leq j< m$, we have%
\begin{equation}
\left\Vert \nabla ^{j}u\right\Vert _{L^{p}}\leq C\left\Vert
\nabla^{m}u\right\Vert _{L^{r}}^{\alpha}\left\Vert u\right\Vert
_{L^{q}}^{1-\alpha}  \label{Interpolation_inequality}
\end{equation}%
where
\begin{equation}
\frac{1}{p}=\frac{j}{n}+\alpha\left( \frac{1}{r}-\frac{m}{n}\right) + \frac{%
1-\alpha}{q}
\end{equation}%
for $\alpha$ satisfying $j/m\leq \alpha \leq 1$ and $C$ is a positive
constant depending only on $n,\, m,\,j,\,q,\,r$ and $\alpha$.
 There
are the following exceptional cases:
\begin{enumerate}
\item If $j=0,\,rm<n$ and $q=\infty $, then we made the additional
assumption that either $u(x)\rightarrow 0$ as $|x|\rightarrow \infty $ or $%
u\in L^{q^{\prime }}$ for some $0<q^{\prime }<\infty .$

\item If $1<r<\infty $ and $m-j-n/r$ is a nonnegative integer, then (\ref{Interpolation_inequality}) holds
only for $j/m\leq \alpha< 1$.  
\end{enumerate}
\end{lemma}
We recall the the Sobolev interpolation of the  Gagliardo--Nirenberg inequality. See \cite[Lemma A.1]{Guo_Wang_2012} for the proof. 
\begin{lemma}\label{Lemma_Sobol_Ga_Ni}
Let $2\leq p\leq +\infty$ and $0\leq m,\alpha,\leq \ell$; when $p=\infty$, we require further that $m\leq \alpha+1$ and $\ell\geq \alpha+2$. Then, we have that for any $f\in C_0^\infty(\R^3)$, 
\begin{equation}
\left\Vert \nabla ^{\alpha}f\right\Vert _{L^{p}}\leq C\Vert \nabla^\ell f\Vert
_{L^{2}}^{\theta} \left\Vert
\nabla^{m}f\right\Vert _{L^{2}}^{1-\theta} \label{Sobolev_Gagl_Ni_Interpolation_ineq_Main} 
\end{equation}%
where $0\leq \theta \leq 1$ and $\alpha$ satsify
\begin{equation}
\frac{1}{p}=\frac{\alpha}{3}+\left( \frac{1}{2}-\frac{\ell}{3}\right)\theta+\left( \frac{1}{2}-\frac{m}{3}\right)(1-\theta) .
\end{equation}%
\end{lemma}
\noindent We also recall the following Sobolev interpolation inequality. 
\begin{lemma}[See Lemma A.4 in~\cite{Guo_Wang_2012}]\label{Lemma_gamma_Interpo}
 Let $\gamma\geq 0$ and $\ell\geq 0$, then we have 
\begin{equation}\label{Fractional_Gag_Nirenberg}
\Vert \nabla ^{\ell}f\Vert _{L^{2}}\leq C\Vert
\nabla^{\ell+1}f\Vert _{L^{2}}^{1-\theta}\Vert \Lambda^{-\gamma} f\Vert
_{L^{2}}^{\theta}   \qquad \text{where}\qquad \theta=\frac{1}{\ell+\gamma+1}. 
\end{equation}
\end{lemma}
\end{appendices}

\section*{Acknowledgements}
The author would like to warmly thank Vanja Nikoli\'c for sharing her many insights about this problem and for the many discussions on the  first draft of this work.  


\end{document}